\documentclass[12pt]{amsart}

\usepackage{amssymb}
\usepackage{amsmath}
\usepackage{amscd}
\usepackage{float}
\usepackage{graphicx}
\usepackage{amsfonts}
\usepackage{pb-diagram}
\usepackage{eufrak}

\newtheorem{thm}{Theorem}

\newtheorem{lem}{Lemma}
\newtheorem{prop}{Proposition}
\newtheorem{defn}{Definition}
\newtheorem{rem}{Remark}

\begin{document}
\title{An adelic extension of the Jones polynomial}

\author{J. Juyumaya}
\address{Departamento de Matem\'aticas, Universidad de Valpara\'{\i}so \\
Gran Breta\~na 1091, Valpara\'{\i}so, Chile.}
\email{juyumaya@uvach.cl}

\author{S. Lambropoulou}
\address{ Departament of Mathematics,
National Technical University of Athens,
Zografou campus, GR-157 80 Athens, Greece.}
\email{sofia@math.ntua.gr}
\urladdr{http://www.math.ntua.gr/$\tilde{~}$sofia}

\thanks{Both authors were partially supported by Fondecyt (grant 1085002), NTUA and  Dipuv.}



\date{}
\maketitle
\begin{abstract}
In this paper we represent the classical braids in the Yokonuma--Hecke and the adelic Yokonuma--Hecke algebras.  More precisely, we define the completion of the framed braid group and we introduce the adelic Yokonuma--Hecke algebras, in analogy to the $p$--adic framed braids and the $p$--adic Yokonuma--Hecke algebras introduced in \cite{jula,jula2}. We further construct an adelic Markov trace, analogous to the $p$--adic Markov trace constructed in \cite{jula2}, and using the traces in \cite{ju} and the adelic  Markov trace we define topological invariants of classical knots and links, upon imposing some condition. Each invariant satisfies a cubic skein relation coming from the Yokonuma--Hecke algebra.
\end{abstract}

\section{Introduction}

The classical braid group on $n$ strands $B_n$ is generated by the elementary braids $\sigma_1 , \ldots , \sigma_{n-1},$ under the defining {\it braid relations}:
$$
\sigma_{i}\sigma_{i+1}\sigma_{i}= \sigma_{i+1}\sigma_{i}\sigma_{i+1}\qquad \text{and}\qquad
\sigma_i\sigma_j = \sigma_j\sigma_i \quad \text{for}\quad \vert i-j\vert >1.
$$
Geometrically, $\sigma_i$ is a positive crossing between the $i$th and the $(i+1)$st strand and $\sigma_i^{-1}$ is the opposite crossing. The operation in $B_n$ corresponds to the concatenation of two braids and the braid relations reflect allowed topological moves. Closing a braid $\beta$ by joining with simple arcs the corresponding top and bottom endpoints of $\beta$ gives rise to an oriented knot or link, denoted $\widehat{\beta}$.  By the classical Alexander theorem, an oriented knot or link can be also isotoped to the closure of a braid. {\it Isotopy} is the notion of topological equivalence for knots and links. Further, by the classical Markov theorem, isotopy classes of oriented knots or links are in bijective correspondence with equivalence classes of  braids in $\cup_nB_n$ under the two moves:
\vspace{.07in}

{\it (i) Conjugation in $B_n$:} \ $\alpha \beta \sim \beta \alpha \qquad \text{and}\qquad$ {\it (ii) Markov move:} $\alpha \sim \alpha{\sigma_n}^{\pm 1}, \ \alpha \in B_n$
\vspace{.07in}

Using the above and Ocneanu's Markov trace on the Iwahori--Hecke algebra of type A, $H_n(q)$, V.F.R. Jones constructed in \cite{jo} the 2--variable Jones or HOMFLYPT polynomial, a new isotopy invariant of oriented knots and links. The algebra $H_n(q)$ can be described naturally as a quotient of the group algebra ${\Bbb C} B_n$ over the quadratic relations:
\begin{equation}\label{quadrhecke}
g_i^2 = (q-1) g_i + q \qquad \mbox{for all $i$}
\end{equation}

 The  Yokonuma--Hecke algebra ${\rm Y}_{d,n}(u), \  d\in \Bbb N,$ is a similar algebraic object and has a natural topological interpretation as quotient of the modular framed braid group algebras ${\Bbb C} {\mathcal F}_{d,n}$ (classical framed braids with framings modulo $d$) over certain quadratic relations, see Eqs. \ref{quadr}. Originally, the algebras ${\rm Y}_{d,n}(u)$ were  introduced by T. Yokonuma \cite{yo} in the representation theory of finite Chevalley groups and they are natural generalizations of the Hecke algebras $H_n(q)$. Indeed, for  $d=1$ the algebra ${\rm Y}_{1,n}(u)$ coincides with the algebra $H_n(q)$. In the above topological interpretation, $d=1$ means all framings zero, so the algebra ${\rm Y}_{1,n}(u)$ is really related to classical braids (with no framings). In \cite{ju} Juyumaya constructed a Markov trace on the algebras ${\rm Y}_{d,n}(u)$, which for $d=1$ coincides with the Ocneanu trace. Further, in \cite{jula} the authors introduced the $p$--adic framed braids and the $p$--adic Yokonuma--Hecke algebras, while in \cite{jula2} they constructed a $p$--adic Markov trace. This was used, together with the trace in \cite{ju}, in order to construct Jones--type isotopy invariants of framed links, upon imposing a certain $E$--condition to the trace parameters, according to the Markov braid equivalence. Finally, in \cite{jula3} they constructed a monoid representation of the singular braid monoid to ${\rm Y}_{d, n}(u)$. Then the trace of \cite{ju} is also a Markov trace on the singular braid monoid, so Jones--type invariants for singular knots were constructed, assuming the $E$--condition.

\smallbreak

In the present paper we first relate the  Yokonuma--Hecke algebras ${\rm Y}_{d,n}(u)$, for $d\neq 1$, to classical knots and links via a natural homomorphism of the classical braid group (Eq.~\ref{repbn}). We further define {\it the completion of the framed braid group} and we introduce {\it the adelic Yokonuma--Hecke algebra} (Section \ref{adelicbrrep}), into which the classical braid group  maps also homomorphically. Then, using the Markov traces in \cite{ju} (for different values of $d$) we construct an infinite family of 2--variable Jones--type invariants of oriented classical knots and links (Section \ref{adelicjones}), upon imposing the $E$--condition (Section \ref{econdition}).
 We also construct, in analogy to \cite{jula2}, an adelic Markov trace (Section \ref{adelictr}), which we use for defining an isotopy invariant of classical knots and links, an {\it  adelic extension of the 2--variable Jones polynomial} (Section~\ref{adelicjones}), upon imposing the $E$--condition. The $E$--condition and the $E$--system are discussed in Section~\ref{econdition}. As far as the braid generators are concerned, the first `closed' relations in the Yokonuma--Hecke algebra are {\it cubic relations} (Eqs. \ref{cubic1}), which also pass to the level of each  invariant in the form of a {\it cubic skein relation} (Eq. \ref{skein}).

\smallbreak

The above--mentioned results are relatively new and computations with the algebras ${\rm Y}_{d,n}(u)$ are very complicated. We are now in the process of creating a computing package. Yet, we believe that our invariants are different from the HOMFLYPT polynomial, mainly for the following reasons. Firstly, the differences of the two algebras, $H_n(q)$ and ${\rm Y}_{d,n}(u)$, and of their quadratic relations. The structure of the Yokonuma--Hecke algebra is much more subtle and complicated than that of the classical Hecke algebra, even made `framed'. In the latter case the quadratic relations Eq.~\ref{quadrhecke} would remain the same as in $H_n(q)$, hence we would talk about the framed HOMFLYPT polynomial. Secondly, the appearance of the cubic relations on the Yokonuma--Hecke algebra level, as the first `closed' relations of the braid generators, which induce a cubic skein relation for each  invariant. Finally, the mere fact that we needed to impose the $E$--condition to the trace parameters in order to obtain a knot invariant, something not needed in the case of the Ocneanu trace. In fact, the trace in \cite{ju} is the first Markov trace in the literature that does not rescale directly in order to yield knot invariants. (Even the fact that the complicated $E$--system has non--trivial solutions was a surprise to us, see Section~\ref{econdition} and \cite{jula2}).

In an effort to keep this paper light we omit some technical details, which are mostly to be found in \cite{jula2}.

\smallbreak

The Yokonuma--Hecke algebras are very versatile algebraic objects, in the sense that they can be used for completely different topological interpretations: to framed braids, to classical braids, to singular braids and, most recently, to transversal and virtual braids, and they comprise the only examples we know of algebras having this property.

\section{An adelic representation of the braid group}\label{adelicbrrep}

\subsection{\it The Yokonuma--Hecke algebra}

Fix $u \in {\Bbb C}\backslash \{0,1\}$. Given two  positive integers $d$ and $n$, we denote ${\rm Y}_{d,n} =  {\rm Y}_{d,n}(u)$ the Yokonuma--Hecke algebra, which is a unital associative algebra over ${\Bbb C}$,  defined by the generators:
$$
1, g_1, \ldots, g_{n-1}, t_1, \ldots, t_{n}
$$
subject to the following relations:
\begin{equation}\label{modular}
\begin{array}{rcll}
g_ig_j & = & g_jg_i & \mbox{for $\vert i-j\vert > 1$}\\
g_ig_jg_i & = & g_jg_ig_j & \mbox{for $ \vert i-j\vert = 1$} \\
t_i t_j & =  &  t_j t_i &  \mbox{for all $ i,j$} \\
t_j g_i & = & g_i t_{s_i(j)} & \mbox{for all $ i,j$}\\
t_j^d   & =  &  1 & \mbox{for all $j$}
\end{array}
\end{equation}
where $s_i(j)$ is the result of applying the transposition $s_i=(i, i+1)$ to $j$ (in particular $t_i g_i  =  g_i t_{i+1}$  and  $ t_{i+1}g_i  = g_i t_i $), together with the extra quadratic relations:
\begin{equation}\label{quadr}
g_i^2 = 1 + (u-1) \, e_{d,i} - (u-1) \, e_{d,i} \, g_i\qquad \mbox{for all $i$}
\end{equation}
 where
\begin{equation}\label{edi}
e_{d,i} :=\frac{1}{d}\sum_{m=0}^{d-1}t_i^m t_{i+1}^{-m}.
\end{equation}
>From the defining relations Eqs.~\ref{modular} and~\ref{quadr} it is easy to check that the elements $e_{d,i}$ are idempotents and that they satisfy the following relations (compare with Lemma~1 in \cite{jula3}).
\begin{equation}\label{edirels}
\begin{array}{rcll}
e_ie_j & = & e_j e_i & \\
e_i g_j & = & g_je_i & \mbox{for  $j=i$  and for $\vert i - j \vert >1$} \\
e_j g_ig_j & =  &  g_ig_je_i &  \mbox{for $\vert i - j \vert =1$}. \\
\end{array}
\end{equation}

\begin{rem}\label{rem1}\rm
For all $1\leq i\leq  n$, let $C_{d,i}=\{1, t_i, t_i^2, \ldots , t_i^{d-1}\}$ denote the cyclic group containing all possible framings modulo $d$ of the $i$th strand of a framed braid on $n$ strands. Notice that $C_{d,i} \simeq {\Bbb Z}/d{\Bbb Z}$ for all $i$. We also define  the group $H:=C_{d,1}\times C_{d,2}\times \ldots \times C_{d,n} \simeq ({\Bbb Z}/d{\Bbb Z})^n$. From the defining relations among the $t_i$'s we deduce that the groups $C_{d,i}$ and $H$ can be regarded  inside ${\rm Y}_{d,n}$.
\end{rem}

  {\it The modular framed braid group} ${\mathcal F}_{d,n}$ contains framed braids on $n$ strands, but  with framings modulo $d$. It is generated by the braiding generators $\sigma_i$ and the framing generators $t_1,\ldots,t_n$, where $t_j$ stands for the identity braid with framing one on the $j$th strand and framing zero on the other strands.  Corresponding the braiding generators $\sigma_i$  to the algebra generators $g_i$, relations Eq.~\ref{modular} furnish a presentation for ${\mathcal F}_{d,n}$ and  the  Yokonuma--Hecke algebra ${\rm Y}_{d,n}$ is a quotient of the modular framed braid group algebra ${\Bbb C} {\mathcal F}_{d,n}$  over the quadratic relations Eq.~\ref{quadr}.  The elements $e_{d,i}$ are in the algebra ${\Bbb C} {\mathcal F}_{d,n}$ as well as in the quotient algebra ${\rm Y}_{d,n}$ and they are expressions of the framing generators $t_i, t_{i+1}$.

\subsection{\it The adelic Yokonuma--Hecke algebra}

Let  ${\Bbb N}$ denote the set of positive integers regarded as a directed set with the usual order. Let also  ${\Bbb N}^{\sim}$ denote the directed set of positive integers regarded with respect to the partial order defined by the divisibility relation. The notation $d\vert d^{\prime}$ means $d$ divides $d^{\prime}$.

For $d$, $d^{\prime}\in {\Bbb N}$ with $d\vert d^{\prime}$ we consider the natural connecting ring homomorphism $\rho ^{d^{\prime}}_d $, defined in \cite{jula2}, Eq.~1.17:
\begin{equation}\label{rhoy}
\rho ^{d^{\prime}}_d :{\rm Y}_{d^{\prime},n} \longrightarrow
{\rm Y}_{d,n}
\end{equation}

More precisely, we denote  $\vartheta^{d^{\prime}}_d$ the natural epimorphism:
\begin{equation}\label{varthetad}
\begin{array}{cccc}
 \vartheta^{d^{\prime}}_d : & {\Bbb Z}/d^{\prime}{\Bbb Z} & \longrightarrow & {\Bbb Z}/d{\Bbb Z}\\
 & m  & \mapsto &   m \,(\text{mod}\, d)
\end{array}
\end{equation}
 The inverse limit $\widehat{{\Bbb Z}}$ of the inverse system of groups $({\Bbb Z}/d{\Bbb Z}, \vartheta^{d^{\prime}}_d)$ indexed by ${\Bbb N}^{\sim}$ is called {\it the completion of ${\Bbb Z}$}:
$$
\widehat{{\Bbb Z}}=\varprojlim_{d\in {\Bbb N}^{\sim}} {{\Bbb Z}/d{\Bbb Z}}
$$

Our references for inverse limits are mainly \cite{riza} and \cite{wi}.

\smallbreak
By componentwise multiplication, epimorphism Eq.~\ref{varthetad} defines the epimorphism:
\begin{equation}\label{varpid}
\begin{array}{cccc}
\varpi^{d^{\prime}}_{d} : & ({\Bbb Z}/d^{\prime}{\Bbb Z})^n  & \longrightarrow &
 ({\Bbb Z}/d{\Bbb Z})^n
 \end{array}
\end{equation}
 Extension to the $B_n$--part by the identity map yields the epimorphism:
\begin{equation}\label{varpidid}
\begin{array}{cccc}
  \varpi^{d^{\prime}}_{d} \cdot {\rm id} : &  {\mathcal F}_{d^{\prime},n}  & \longrightarrow &  {\mathcal F}_{d,n}
\end{array}
\end{equation}
\begin{defn}\rm
 {\it The completion ${\mathcal F}_{\infty , n}$ of the framed braid group ${\mathcal F}_n$} is defined as the inverse limit of the inverse system of groups $({\mathcal F}_{d,n} , \, \varpi^{d^{\prime}}_{d}\cdot {\rm id})$:
$$
{\mathcal F}_{\infty , n} := \varprojlim_{d\in {\Bbb N}^{\sim}} {\mathcal F}_{d,n}
$$
\end{defn}

The linear extension of map Eq.~\ref{varpidid} yields an algebra epimorphism:
\begin{equation}\label{varrho}
\varrho ^{d^{\prime}}_d : {\Bbb C} {\mathcal F}_{d^{\prime},n}\longrightarrow {\Bbb C} {\mathcal F}_{d,n}
\end{equation}

\begin{rem}\rm
The braid group $B_n$ acts on $\widehat{{\Bbb Z}}^{\,n}$ by permuting the factors, so we may consider the group $\widehat{{\Bbb Z}}^{\,n}\rtimes B_n$. It is easy to construct an isomorphism between the groups $\widehat{{\Bbb Z}}^{\,n}\rtimes B_n$ and ${\mathcal F}_{\infty , n}$ (proof analogous to Theorem~1 in \cite{jula}). We note, though, that this isomorphism does not carry through on the level of the algebras ${\Bbb C} (\widehat{{\Bbb Z}}^{\,n}\rtimes B_n)$ and $\varprojlim_{d\in {\Bbb N}^{\sim}} {\Bbb C} {\mathcal F}_{d,n}$   (see \cite{jula2} for more details).
\end{rem}

Passing now to the quotient algebras by relations Eq.~\ref{quadr} we obtain the following algebra epimorphism:
\begin{equation}\label{rhoy}
\rho ^{d^{\prime}}_d :{\rm Y}_{d^{\prime},n} \longrightarrow
{\rm Y}_{d,n}
\end{equation}

\begin{defn}\rm
 {\it The adelic Yokonuma--Hecke} algebra ${\rm Y}_{\infty,n}(u) = {\rm Y}_{\infty,n}$ is defined as the inverse limit of the inverse system of rings $({\rm Y}_{d,n}, \rho_d^{d^{\prime}})$ indexed by ${\Bbb N}^{\sim}$:
$$
  {\rm Y}_{\infty,n}  =  \varprojlim_{d\in {\Bbb N}^{\sim}} {\rm Y}_{d,n}
$$
\end{defn}

Hence, elements in ${\rm Y}_{\infty,n}$ are infinite sequences of elements in the algebras ${\rm Y}_{d,n}$, for $d\in {\Bbb N}^{\sim}$, which are coherent in the sense of maps Eqs.~\ref{varthetad} --~\ref{rhoy}.
Moreover, the definition of the connecting  maps $\rho ^{d^{\prime}}_d$ do not involve the elements  $g_i$, so we shall  denote also by $g_i$ the elements in  ${\rm Y}_{\infty,n} $ corresponding to the infinite
 constant sequence  $\left( g_i\right)$.

\smallbreak
For all $0\leq i\leq n-1$  define now the groups $H_{d,i}$ as follows:
$$
H_{d, i} =\{1, t_it_{i+1}^{-1}, t_i^2t_{i+1}^{-2}, \ldots , t_i^{d-1}t_{i+1}\}
$$
In this notation, the element $e_{d,i}$ in Eq.~\ref{edi} is the  average of the elements of the group  $H_{d,i}$:
$$
e_{d,i} = \frac{1}{d}\sum_{x\in H_{d,i}}x
$$
Then  $\rho ^{d^{\prime}}_d(H_{d^{\prime},i})= H_{d,i}$ for all $d\vert d^{\prime}$. Hence, we deduce the following result.

\begin{lem} For all $i$ and for $d$, $d^{\prime}$ such that $d\vert d^{\prime}$, we have:
$$
\rho ^{d^{\prime}}_d (e_{d^{\prime},i}) = e_{d,i}.
$$
\end{lem}

We shall denote by $e_i$  the sequence $\left( e_{d, i} \right)_{d\in {\Bbb N}^{\sim}}$ in
${\rm Y}_{\infty,n}$\,:
\begin{equation}\label{ei}
e_i := \left( e_{d, i} \right)_{d\in {\Bbb N}^{\sim}}
\end{equation}

It follows easily from relations Eqs.~\ref{edirels} and~\ref{quadr} that the adelic elements $e_i$ satisfy the following relations (compare with Proposition~10 and Theorem~3 in \cite{jula}).
\begin{prop}
For all $i$ the following relations hold  in ${\rm Y}_{\infty,n}$:
\begin{enumerate}
\item $e_ie_j = e_j e_i$
\item $e_i g_j = g_je_i$  for  $j=i$  and for $\vert i - j \vert >1$
\item $e_j g_ig_j = g_ig_je_i$ for $\vert i - j \vert =1$
\item $g_i^ 2 = 1 + (u-1)e_ig_i - (u-1)g_i.$
\end{enumerate}
\end{prop}

\subsection{\it Representing the braid group}

The defining relations  of ${\rm Y}_{d,n}$ imply that the map:
\begin{equation}\label{repbn}
\begin{array}{cccc}
\flat_{d,n} : &  B_n &  \longrightarrow &  {\rm Y}_{d,n} \\
& \sigma_i & \mapsto & g_i
\end{array}
\end{equation}
defines a representation of the classical braid group $B_n$ in ${\rm Y}_{d,n}$. Under this representation the generators $g_i$ of the algebra ${\rm Y}_{d,n}$ correspond to the braid generators $\sigma_i$. The generators $t_j$, though,  loose their initial topological interpretation as framing generators and they are just considered formally as elements in the algebra.

\smallbreak
Further, for all $d, d^{\prime}, d^{\prime\prime}$ such that $d\vert d^{\prime}$ and $d^{\prime}\vert d^{\prime\prime}$ we have the following commutative diagram:
\begin{equation}\label{rep}
\begin{diagram}
\node{\cdots}
\node{B_n}
\arrow{w,t}{} \arrow{s,l}{\flat_{d,n}} \node{B_n}\arrow{w,t}{{\rm Id}}
\arrow{s,l}{\flat_{d^{\prime},n}}
\node{B_n}
\arrow{w,t}{{\rm Id}} \arrow{s,l}{\flat_{d^{\prime\prime},n}} \node{\cdots}\arrow{w,t}{}
\\
\node{\cdots}
\node{{\rm Y}_{d,n}}
\arrow{w,t}{}
\node{{\rm Y}_{d^{\prime},n}}\arrow{w,t}{\rho_d^{d^{\prime}}}
\node{{\rm Y}_{d^{\prime\prime},n}}
\arrow{w,t}{\rho_{d^{\prime}}^{d^{\prime\prime}}}
\node{\cdots}\arrow{w,t}{}
\end{diagram}
\end{equation}

By taking inverse limits in the above diagram we obtain the following representation of the classical braid group  $B_n$ in the adelic Yokonuma--Hecke algebra:
\begin{equation}\label{adelicrepbn}
\flat_{\infty,n} : B_n
 \longrightarrow {\rm Y}_{\infty,n}
\end{equation}
where:
$$
\flat_{\infty,n}:=  \varprojlim_{d\in {\Bbb N}^{\sim}} \flat_{d,n}
$$

\section{An adelic Markov trace}\label{adelictr}

\subsection{\it The modular Markov trace ${\rm tr}_{d}$}

It is known that the Yokonuma--Hecke algebra supports a Markov trace \cite{ju}. More precisely, for fixed $d$ we consider the inductive system $({\rm Y}_{d,n})_{n\in {\Bbb N}}$ associated
to the natural inclusions ${\rm Y}_{d,n} \subset {\rm Y}_{d,n+1}$  for all  $n\in{\Bbb N}$. Let ${\rm Y}_{d, \infty}$ be the corresponding inductive limit. In \cite{ju} the following theorem is proved.

\begin{thm}[Juyumaya, 2004]\label{trace}
Let $z$, $x_1$, $\ldots$, $x_{d-1}\in {\Bbb C}$ and let $d$ be a positive integer. For all $n\in{\Bbb N}$
there exists a unique linear map ${\rm tr}_{d}=\left({\rm tr}_{d,n}\right)_{n\in {\Bbb N}} $:
$$
{\rm tr}_{d} : {\rm Y}_{d, \infty} \longrightarrow {\Bbb C}
$$
 satisfying the rules:
 $$
\begin{array}{rcll}
{\rm tr}_{d,n}(ab) & = & {\rm tr}_{d,n}(ba) & \\
{\rm tr}_{d,n}(1) &  = & 1 & \\
{\rm tr}_{d,n+1}(ag_n) & = & z\,{\rm tr}_{d,n}(a)& \qquad (a \in {\rm Y}_{d,n})\\
{\rm tr}_{d, n+1}(at_{n+1}^m) & = & x_m{\rm  tr}_{d,n}(a) & \qquad ( a \in {\rm Y}_{d,n}\,, \, 1\leq m\leq d-1).
\end{array}
$$
\end{thm}

The proof of Theorem~\ref{trace} rests on the fact that the algebra ${\rm Y}_{d,n+1}$ admits an inductive linear basis, where either $g_n$ or $t_{n+1}^m$ appears at most once.
Note that, for $d=1$ the trace restricts to the first three rules and it coincides with Ocneanu's trace on the Iwahori--Hecke algebra, which was used in \cite{jo} to construct the 2--variable or HOMFLYPT Jones polynomial for oriented classical knots and links.

\subsection{\it  The adelic Markov trace ${\rm tr}_{\infty}$}

Let $R$ be the polynomial ring ${\Bbb C}[z]$ and let $R[X_{d}]$ be the polynomial ring  with coefficients
 in $R$ and  variables $x_a$, where $a\in {\Bbb Z}/d{\Bbb Z}$. Let also $d\vert d^{\prime}$. The natural map
 $
 x_a\mapsto x_{b}
 $
 where $b:= \vartheta^{d^{\prime}}_d(a)$ (recall Eq.~\ref{varthetad}), defines a ring epimorphism:
\begin{equation}\label{xi}
\xi_{d}^{d^{\prime}} :R\left[X_{d^{\prime}} \right] \longrightarrow
R\left[X_{d}\right]
\end{equation}

We now have the following result (compare with Lemma 7 in \cite{jula2}).

\begin{lem}
The family $\left(R[X_d],\xi_d^{d^{\prime}}\right)$ indexed by ${\Bbb N}^{\sim}$,  is an inverse system.
\end{lem}
We shall then consider the inverse limit:
$$
\varprojlim_{d\in {\Bbb N}^{\sim}}  R[X_d]
$$
 Notice that  $\varprojlim_{d\in {\Bbb N}^{\sim}}  R[X_d]$ can be regarded as the polynomial ring
over ${\Bbb C}$ in the variables $z$ and  $x_{\alpha}$, where $\alpha\in \widehat{\Bbb Z}$. The ring $\varprojlim_{d\in {\Bbb N}^{\sim}} R[X_d]$ turns out to be an integral domain.

Now, for all $n\in {\Bbb N}$ and for all $d, d^{\prime}, d^{\prime\prime}$ such that $d\vert d^{\prime}$ and $d^{\prime}\vert d^{\prime\prime}$, we have the following commutative diagram (compare with Lemma 6\cite{jula2}):
\begin{equation}\label{tra}
\begin{diagram}
\node{\cdots}
\node{{\rm Y}_{d,n}}
\arrow{w,t}{} \arrow{s,l}{{\rm tr}_{d,n}} \node{{\rm Y}_{d^{\prime},n}}\arrow{w,t}{\rho_d^{d^{\prime}}}
\arrow{s,l}{{\rm tr}_{d^{\prime},n}}
\node{{\rm Y}_{d^{\prime\prime},n}}
\arrow{w,t}{\rho_{d^{\prime}}^{d^{\prime\prime}}} \arrow{s,l}{{\rm tr}_{d^{\prime\prime},n}} \node{\cdots}\arrow{w,t}{}
\\
\node{\cdots}
\node{R\left[X_d\right] }
\arrow{w,t}{}
\node{R\left[X_{d^{\prime}}\right] }\arrow{w,t}{\xi_d^{d^{\prime}}}
\node{R\left[X_{d^{\prime\prime}}\right] }
\arrow{w,t}{\xi_{d^{\prime}}^{d^{\prime\prime}}}
\node{\cdots}\arrow{w,t}{}
\end{diagram}
\end{equation}

Finally, we note that there are natural inclusions ${\rm Y}_{\infty,n} \subset {\rm Y}_{\infty,n+1}$, for all  $n\in{\Bbb N}$. Let
$$
{\rm Y}_{\infty} := \varinjlim_{n\in {\Bbb N}} {\rm Y}_{\infty, n}
$$
the associated inductive limit. We then have the following.

\begin{thm}\label{ptrace}
There  exists a unique linear Markov trace ${\rm tr}_{\infty} =\left({\rm tr}_{\infty,n}\right)_{n\in{\Bbb N}}$,
$$
{\rm tr}_{\infty} :  {\rm Y}_{\infty } \longrightarrow  \varprojlim_{d\in {\Bbb N}^{\sim}}  R[X_d]
$$
such that
$$
\begin{array}{rcll}
{\rm tr}_{\infty,n}(ab) & = & {\rm tr}_{\infty,n}(ba) & \\
{\rm tr}_{\infty,n}(1) &  = & 1 & \\
{\rm tr}_{\infty,n+1}(ag_n) & = & z\,{\rm tr}_{\infty,n}(a)& \\
{\rm tr}_{\infty, n+1}(ay_{n+1}) & = & x_y {\rm  tr}_{\infty,n}(a) &
\end{array}
$$
where $a$, $b\in {\rm Y}_{\infty, n}$ and $y_{n+1}$ is the element in $\widehat{{\Bbb Z}}^{\,n+1}$ with $y \in \widehat{{\Bbb Z}}$ in the position $n+1$ and $0$ otherwise, that is: $y_{n+1} = (0,\ldots,0,y)$.
\end{thm}
\begin{proof} It follows immediately from the commutative diagram Eq.~\ref{tra} and from the existence and uniqueness of the traces ${\rm tr}_{d}$.
\end{proof}

\section{The $E$--condition }\label{econdition}

\subsection{}

The representations Eqs.~\ref{repbn} and~\ref{adelicrepbn} of the braid group through the classical and the adelic Yokonuma--Hecke algebras composed with the Markov traces ${\rm tr}_d$ and ${\rm tr}_{\infty}$ of Theorems \ref{trace} and \ref{ptrace} map classical braids to complex polynomials. In view of the Alexander and Markov theorems for classical braids we would like to construct isotopy invariants for classical oriented knots and links. According to the Markov Theorem, such an invariant has to agree on the links $\widehat{\alpha}$, $\widehat{\alpha\sigma_n}$ and $\widehat{\alpha\sigma_n^{- 1}}$, for any $\alpha \in B_n$. Following Jones' construction of the 2--variable Jones polynomial for classical knots \cite{jo}, we will try to define knot isotopy invariants by re--scaling and normalizing the traces ${\rm tr}_{d}$ and the adelic trace ${\rm tr}_{\infty}$. By the equation:
\begin{equation}\label{invrs}
g_i^{-1} = g_i - (u^{-1} - 1)\, e_{d,i} + (u^{-1} - 1)\, e_{d,i} \, g_i
\end{equation}
 we have:
\begin{equation}\label{trinvrs}
{\rm tr}_{d}(\alpha g_n^{-1}) = {\rm tr}_{d}(\alpha g_n) - (u^{-1}-1){\rm tr}_{d}(\alpha e_{d,n})
 + (u^{-1}-1){\rm tr}_{d}(\alpha e_{d,n} g_n).
\end{equation}
In order that the invariant agrees on the closures of the braids $\alpha{\sigma_n}^{-1}$ and $\alpha{\sigma_n}$ we need that ${\rm tr}_{d}(\alpha g_n^{-1})$ factorizes through ${\rm tr}_{d} (\alpha)$, just as ${\rm tr}_{d}(\alpha g_n)$ does.  Indeed, for the first term we have: ${\rm tr}_{d}(\alpha g_n) = z\, {\rm tr}_{d}(\alpha)$. Further:

\begin{equation}\label{egalpha}
{\rm tr}_{d} (\alpha e_{d,n}g_n)=\frac{1}{d}\sum_{m=0}^{d-1}{\rm tr}_{d}(\alpha t_n^mt_{n+1}^{-m} g_n )=\frac{1}{d}\sum_{m=0}^{d-1}z\, {\rm tr}_{d}(\alpha) = z\,{\rm tr}_{d}(\alpha)
\end{equation}
since  ${\rm tr}_{d}(\alpha t_n^mt_{n+1}^{-m} g_n)= {\rm tr}_{d}( \alpha t_n^mg_n t_{n}^{-m})= z\,{\rm tr}_{d}( \alpha t_n^mt_n^{-m}) = z\,{\rm tr}_{d}(\alpha)$. So, we need  that ${\rm tr}_{d}(\alpha e_{d,n})$ also factorizes through ${\rm tr}_{d} (\alpha)$. Unfortunately, we do not have, in general, such a  nice formula for ${\rm tr}_{d}(\alpha e_{d,n})$. The underlying reason on the framed braid level (which is the natural interpretation for elements in ${\rm Y}_{d,n}$) is that $e_{d,n}$ involves the $n$th strand of the braid $\alpha$. Yet, by imposing some conditions on the indeterminates $x_i$ of the trace ${\rm tr}_{d}$ it is possible to have this factorization.

\subsection{\it The $E$--system}

Set  $X_d=\{ x_0,  x_1, \ldots ,x_{d-1}\}$ a set of $d$ complex numbers. We shall say that $X_d$  satisfies the {\it $E$--condition} if the $x_i$'s  are solutions of the following non--linear system of $d-1$ equations:
\begin{equation}\label{Esystem}
\begin{array}{ccc}
{\rm E}_d^{(1)} & = &  {\rm x}_1 {\rm E}_d^{(0)}\\
{\rm E}_d^{(2)} & = &  {\rm x}_2 {\rm E}_d^{(0)}\\
& \vdots & \\
{\rm E}_d^{(d-1)} & = &  {\rm x}_{d-1}{\rm E}_d^{(0)}
\end{array}
\end{equation}
where ${\rm E}_{d}^{(m)}$ is the   polynomial in variables ${\rm x}_1,\ldots ,{\rm x}_{d-1}$ defined as:
\begin{equation}\label{Epoly}
{\rm E}_{d}^{(m)} =\sum_{s=0}^{d-1}{\rm x}_{m+s}{\rm x}_{d-s}
\end{equation}
where,  by definition, ${\rm x}_0 = {\rm x}_d = 1$, and the sub--indices are regarded modulo $d$.
 We shall refer to the system above as the {\it $(E,d)$--system} or simply the {\it $E$--system}.
 For example, in the case $d=3$ we have the $E$--system:
$$
\begin{array}{lcr}
{\rm x}_1 + {\rm x}_2^2 & = & 2 {\rm x}_1^2 {\rm x}_2 \\
{\rm x}_1^2 + {\rm x}_2 & = & 2 {\rm x}_1 {\rm x}_2^2
\end{array}
$$

We then have the following result (compare with Theorem~6 in \cite{jula2}).

\begin{thm}\label{thme}
If $X_{d, S}$ is a solution of the $E$--system then for all $\alpha \in {\rm Y}_{d,n}$ we have:
$$
{\rm tr}_d(\alpha e_{d,n}) = {\rm tr}_d(\alpha)\, {\rm tr}_d( e_{d,n}).
$$
\end{thm}
For the proof of Theorem~\ref{thme} we need to consider all different cases for $\alpha$ being an element in the inductive basis of ${\rm Y}_{d,n}(u)$. See \cite{jula2} for details.

\smallbreak
We still need to establish, of course, that the set of solutions of the $E$--system is non--empty. For $a\in {\Bbb Z}/d{\Bbb Z}$ we denote  ${\rm exp}_a$ the exponential character of the group ${\Bbb Z}/d{\Bbb Z}$, that is:
$$
{\rm exp}_a (k): = \cos \frac{2\pi  a k}{d} + i \sin \frac{2\pi  a k}{d}\qquad (k\in {\Bbb Z}/d{\Bbb Z}).
$$
\begin{thm}[G\'{e}rardin, 2009]
The solutions of system Eq.~\ref{Esystem} are parametrized by the non--empty subsets $S$ of ${\Bbb Z}/d{\Bbb Z}$. More precisely, a subset $S$ defines the solution $X_{d, S}=\{x_0, x_1, \ldots ,x_{d-1}\}$,  where:
$$
x_k = \frac{1}{\vert S\vert}\sum_{s\in S}{\rm exp}_s(k)\qquad (0\leq k\leq d-1).
$$
\end{thm}

\begin{proof}
See Appendix in \cite{jula2}.
\end{proof}

Let $X_{d, S}$ be a solution of the $E$--system. A direct computation yields that the value of the ${\rm tr}_d$ on $e_{d,i}$ (with respect to $X_{d, S}$) is:
\begin{equation}\label{valedi}
{\rm tr}_d\left( e_{d,i}\right) = \frac{ 1}{\vert S\vert}\qquad (1\leq i\leq n-1).
\end{equation}

For a thorough discussion and full proofs related to the $E$--condition and the $E$--system we refer the reader to \cite{jula2}.

\subsection{\it }

For  $d\vert d^{\prime}$ we denote ${\rm s}_{d^{\prime}}^d$ a section map of the natural epimorphism $\vartheta_d^{d^{\prime}}$ of Eq.~\ref{varthetad}.
  By taking a section  ${\rm s}_{d^{\prime}}^d$ any solution of the $(E, d)$--system can be lifted trivially to a  solution of the  $(E, d^{\prime})$--system. Indeed: If $X_{d, S}$ is a solution of the $(E, d)$--system, then  $X_{d^{\prime}, S^{\prime}}$ is a solution of the $(E, d^{\prime})$--system, where $S^{\prime} :={\rm s}_{d^{\prime}}^d(S) $. A more interesting lifting  can be constructed as follows. Define
  $S^d_{d^{\prime}} =\{{\rm s}^d_{d^{\prime}}(a) + b\,;a\in S, \, b\in\ker\vartheta^{d^{\prime}}_d\}$. Then we define the lifting  $X_{d, d^{\prime}, S}$ of $X_{d, S}$ as:
\begin{equation}\label{lif}
  X_{d, d^{\prime}, S} : = X_{d^{\prime}, S^d_{d^{\prime}}}\qquad (S\subseteq {\Bbb Z}/d{\Bbb Z})
\end{equation}
Notice that $\vert X_{d, d^{\prime}, S} \vert = \vert S\vert d^{\prime}/d $  and $X_{d,d, S} = X_{d, S} $.

\begin{lem}\label{cohlif}
For $d\vert d^{\prime}\vert d^{\prime\prime}$ and $S$ non--empty subset of ${\Bbb Z}/d{\Bbb Z}$ we have:
$$
X_{d, d^{\prime\prime}, S} = X_{ d^{\prime}, d^{\prime\prime}, S^{\prime} }
$$
where $S^{\prime}:=S^d_{d^{\prime}}$.
\end{lem}
\begin{proof}
According to the definition of $X_{S, d}$ it is enough to prove that:
$$
 \left(S^d_{d^{\prime}}\right)^{d^{\prime}}_{d^{\prime\prime}}= S^d_{d^{\prime}}
$$
Now the elements in $\left(S^d_{d^{\prime}}\right)^{d^{\prime}}_{d^{\prime\prime}}$ are in the form
$z:={\rm s}^{d^{\prime}}_{d^{\prime\prime}}(x) + y$, where $x\in S^d_{d^{\prime}}$ and $y\in \ker \vartheta^{d^{\prime\prime}}_{d^{\prime}}$. The element $x$ is in the form $x ={\rm s}^d_{d^{\prime}}(\mu) + \nu $, where $\mu\in S$ and $\nu\in \ker \vartheta^{d^{\prime}}_d$.  So we can re--write $z$ as:
$$
z = {\rm s}^{d^{\prime}}_{d^{\prime\prime}}\left( {\rm s}^d_{d^{\prime}}(\mu) + \nu \right) + y = {\rm s}^{d^{\prime}}_{d^{\prime\prime}}\left( {\rm s}^d_{d^{\prime}}(\mu) \right)+ {\rm s}^{d^{\prime}}_{d^{\prime\prime}}\left(\nu\right) + y = {\rm s}^d_{d^{\prime\prime}}\left(\mu\right)+ {\rm s}^{d^{\prime}}_{d^{\prime\prime}}\left(\nu\right) + y
$$
But ${\rm s}^{d^{\prime}}_{d^{\prime\prime}}\left(\nu\right) + y$ belongs to the $\ker\vartheta^{d^{\prime\prime}}_d$; hence $z\in S^d_{d^{\prime}}$. Thus $\left(S^d_{d^{\prime}}\right)^{d^{\prime}}_{d^{\prime\prime}}= S^d_{d^{\prime}}$.
\end{proof}
We showed that solutions of the $E$--system lift to solutions on the adelic level.

\section{An adelic extension of the Jones polynomial}\label{adelicjones}

\subsection{\it Isotopy invariants from ${\rm tr}_d$}

Given a solution $X_{d, S}$ of the $E$--system, Eq. \ref{trinvrs} can be rewritten as follows, using Theorem~\ref{thme}:
\begin{equation}\label{Etrinvrs}
{\rm tr}_{d}(\alpha g_n^{-1}) =  \frac{z +(u-1)\zeta_{d,S}}{u}\, {\rm tr}_{d}(\alpha)
\end{equation}
where, for all $i$:
$$
\zeta_{d,S} := {\rm tr}_d(e_{d,i}) = \frac{1}{\vert S\vert }.
$$

 Let now ${\mathcal L}$ be the set of oriented links in $S^3$. Recall that by the Alexander theorem every link type may be represented by a closed braid. For the solution $X_{d, S}$ of the $E$--system we wish to define a link isotopy invariant $\Delta_{d, S}$. In order that $\Delta_{d, S}(\widehat{\alpha\sigma_n}) = \Delta_{d, S}(\widehat{\alpha\sigma_n^{- 1}})$, for $\alpha \in B_n$, we apply a re--scaling via the homomorphism:
\begin{equation}\label{delta}
\begin{array}{cccl}
\delta : & B_n & \longrightarrow & {\rm Y}_{d,n} \\
 & \sigma_i  & \mapsto & \sqrt{\lambda}g_i
\end{array}
\end{equation}
where:
$$
 \lambda:= \frac{z -(1-u)\zeta_{d,S}}{uz}
$$
Finally, in order that $\Delta_{d, S}(\widehat{\alpha\sigma_n}) = \Delta_{d, S}(\widehat{\alpha})$ we need to do a normalization. So, we define the following map on the set ${\mathcal L}$.

\begin{defn}\rm
Let $\alpha\in B_n$, any $n$. We define the map $\Delta_{d, S}$ on the closure $\widehat{\alpha}$ of $\alpha$ as follows:
$$
\Delta_{d, S} (\widehat{\alpha}) := \left(\frac{1-\lambda u}{\sqrt{\lambda}(1-u)\zeta_{d,S}}\right)^{n-1}
\left({\rm tr}_d \circ \delta\right)(\alpha)
$$
Equivalently, setting
$$
D:= \frac{1-\lambda u}{\sqrt{\lambda}(1-u)\zeta_{d,S}}
$$
 we can write:
$$
\Delta_{d, S} (\widehat{\alpha}) = D^{n-1}(\sqrt{\lambda})^{\epsilon(\alpha)}{\rm tr}_d(
\flat_{d,n}(\alpha))
$$
where $\epsilon(\alpha)$ is the algebraic sum of the exponents of the $\sigma_i$'s in the braid word $\alpha$ and where $\flat_{d,n}$ was defined in Eq.~\ref{repbn}.
\end{defn}

\begin{thm}\label{invariant}
For any solution $X_{d, S}$ of the $E$--system, $\Delta_{d, S}$ is a 2--variable isotopy invariant  of oriented links in $S^3$, depending on the variables $u,z$.
\end{thm}

\begin{proof}
We need to show that $\Delta_{d, S}$ is well--defined on isotopy classes of oriented links. According to the Markov theorem, it suffices to prove that $\Delta_{d, S}$ is consistent with moves (i) and (ii). From the facts that $\epsilon(\alpha\alpha^\prime) = \epsilon(\alpha^\prime\alpha)$ and
 ${\rm tr}_d(ab)= {\rm tr}_d(ba)$, it follows that $\Delta_{d, S}$ respects move (i).
Let now $\alpha \in B_n$. Then $\alpha\sigma_n\in B_{n+1}$ and $\epsilon(\alpha\sigma_n)= \epsilon(\alpha)+1$. Hence:
$$
\Delta_{d, S}(\widehat{\alpha\sigma_n})
 = D^{n}(\sqrt{\lambda})^{\epsilon(\alpha\sigma_n)}\, {\rm tr}_d(\flat_{d,n}(\alpha \sigma_n))
 = D^{n}(\sqrt{\lambda})^{\epsilon(\alpha)+1}\, {\rm tr}_d(\flat_{d,n}(\alpha)g_n)
 = D \sqrt{\lambda}\, z \, \Delta_{d, S}(\widehat{\alpha})
$$
where we used that ${\rm tr}_d(\flat_{d,n} (\alpha)g_n) = z\, {\rm tr}(\flat_{d,n} (\alpha))$.  Now:
$$
z = \frac{(1-u)\zeta_{d,S}}{1-\lambda u},
$$
 so:
 $$
  D \sqrt{\lambda}\, z  = 1.
  $$
   Therefore, $\Delta_{d, S}(\widehat{\alpha\sigma_n}) = \Delta_{d, S}(\widehat{\alpha})$. Finally, we will prove that $\Delta_{d, S}(\widehat{\alpha\sigma_n^{-1}}) = \Delta_{d, S}(\widehat{\alpha})$.  Indeed:
$$
\Delta_{d, S}(\widehat{\alpha\sigma_n^{-1}})
= D^{n}(\sqrt{\lambda})^{\epsilon(\alpha\sigma_n^{-1})}{\rm tr}_d(\flat_{d,n}(\alpha \sigma_n^{-1}))
= D^{n}(\sqrt{\lambda})^{\epsilon(\alpha)-1}\, {\rm tr}_d(\flat_{d,n}(\alpha)g_n^{-1}).
$$
Resolving $g_n^{-1}$ from Eq.~\ref{invrs} we obtain:
$$
\Delta_{d, S} (\widehat{\alpha\sigma_n^{-1}})
=  D^{n}(\sqrt{\lambda})^{\epsilon(\alpha)-1}\,
\left[ z - (u^{-1}-1)\zeta_{d,S} + (u^{-1}-1)z \right]\, {\rm tr}_d(\flat_{d,n}(\alpha)).
$$
Also, from Theorem \ref{thme} and Eq. \ref{egalpha} we have:
$$
{\rm tr}_d( \flat_{d,n} (\alpha e_{d,n})) = \zeta_{d,S} \,{\rm tr}_d(\flat_{d,n} (\alpha))\quad \text{and} \quad
{\rm tr}_d( \flat_{d,n} (\alpha) e_{d,n}g_n) = z\, {\rm tr}_d(\flat_{d,n} (\alpha)).
$$
Therefore:
$$
\Delta_{d, S} (\widehat{\alpha\sigma_n^{-1}})
=  D^{n}(\sqrt{\lambda})^{\epsilon(\alpha)-1}\, \frac{z+(u-1)\zeta_{d,S}}{u}\, {\rm tr}_d(\flat_{d,n}(\alpha))
= \frac{D}{\sqrt{\lambda}} \, \frac{z+(u-1)\zeta_{d,S}}{u}\,\Delta_{d, S}(\widehat{\alpha}) =
$$
$\Delta_{d, S}(\widehat{\alpha})$.
Hence the proof is concluded.
\end{proof}

We have defined an infinite family of 2--variable isotopy invariants for oriented classical links.

\subsection{\it Computations}

 We shall first give some formulas that are useful for computations.
For powers of $g_i$ we can  deduce by induction the following formulae.

\begin{lem}\label{powers}
Let $m \in {\Bbb Z}, k \in {\Bbb N}$. (i) For $m$ positive, define $\alpha_m = (u-1)
\sum_{l=0}^{k-1}u^{2l}$ if $m=2k$ and $\beta_m = u(u-1)
\sum_{l=0}^{k-1}u^{2l}$ if $m=2k+1$. Then:
\[
g_i^{m} =\left\{
\begin{array}{l} 1 + \alpha_m\, e_{d,i} - \alpha_m \, e_{d,i}\,
g_i \quad \text{ \ if \ } m=2k \\
 g_i - \beta_m\, e_{d,i} + \beta_m \, e_{d,i} g_i \quad \text{ \ if \ } m=2k+1
\end{array}\right.
\]
(ii) For $m$ negative, define
$\alpha'_m = u^{-1}(u^{-1}-1) \sum_{l=0}^{k-1}u^{-2l}$ if $m=-2k$
and $\beta'_m = (u^{-1}-1) \sum_{l=0}^{k-1}u^{-2l}$ if $m=-2k+1$.
Then:
\[
g_i^{m} = \left\{
\begin{array}{l} 1 + \alpha'_m\, e_{d,i} - \alpha'_m \, e_{d,i}\,
g_i \quad \text{ \ if \ } m=-2k \\
 g_i - \beta'_m\, e_{d,i} + \beta'_m \, e_{d,i} g_i \quad \text{ \ if \ } m=-2k+1
\end{array}\right.
\]
\end{lem}

We now proceed with some basic computations.
 Clearly, for the unknot ${\rm O}$, $\Delta_{d, S} ({\rm O})=1$. For the Hopf link and the two trefoil knots we have:

\smallbreak
\noindent $\bullet$ Let ${\rm H}=\widehat{\sigma_1^2}$, the Hopf link. We find ${\rm tr }_d(g_1^2) =  1 + (u+1)(\zeta_{d,S} - z)$ and $\epsilon(\sigma_1^2) = 2$. Then:
$$
\Delta_{d, S}({\rm H}) =\frac{1-\lambda u}{(1-u)\zeta_{d,S}}\sqrt{\lambda}\, \left(1 + (u+1)(\zeta_{d,S} - z)\right) = z^{-1}\sqrt{\lambda}
\left(1 + (u+1)(\zeta_{d,S} - z)\right).
$$

\noindent $\bullet$ Let ${\rm T}=\widehat{\sigma_1^3}$, the right--handed trefoil.  From Lemma~\ref{powers} we have  $g_1^3 = g_1 -u(u-1)e_{d,1} + u(u-1)e_{d,1}g_1$. Hence:
${\rm tr}_d(g_1^3) = z -u(u-1)\zeta_{d,S} + u(u-1)z$. Moreover $\epsilon(\sigma_1^ 3)= 3$.
Then, using that $1-\lambda u = z^ {-1}\zeta_{d,S}(1-u)$, we obtain:
\begin{eqnarray*}
\Delta_{d, S}({\rm T}) & = & D(\sqrt{\lambda})^3\left[(u(u-1) + 1)z-u(u-1)\zeta_{d,S}\right]\\
& = &
\frac{\lambda}{z}\left[(u^2 -u + 1)z - (u^2 -u)\zeta_{d,S}\right].
\end{eqnarray*}

\smallbreak
\noindent $\bullet$ Let, finally, $-{\rm T} = \widehat{\sigma_1^{-3}}$, the left--handed trefoil.  From Lemma~\ref{powers} we have  $g_1^{-3} = g_1 - (u^{-1}-1)(u^{-2}+1)e_{d,1} + (u^{-1}-1)(u^{-2}+1) e_{d,1}g_1$. Hence:
${\rm tr}_d(g_1^{-3}) = z - (u^{-1}-1)(u^{-2}+1) \zeta_{d,S} + (u^{-1}-1)(u^{-2}+1) z$. Moreover $\epsilon(\sigma_1^{-3})= -3$.
Then we obtain:
\begin{eqnarray*}
\Delta_{d, S}(-{\rm T}) & = & D(\sqrt{\lambda})^{-3}\left[(u^{-3} -u^{-2} +u^{-1})z - (u^{-3} -u^{-2} + u^{-1} - 1)\zeta_{d,S}\right],
\end{eqnarray*}
where we recall that $D= \frac{1-\lambda u}{\sqrt{\lambda}(1-u)\zeta_{d,S}}$.

\subsection{\it A cubic skein relation for $\Delta_{d, S}$}

Let $L_{+}$, $L_{-}$, $L_0$ be diagrams of oriented links, which are all identical, except near one crossing, where they differ by the ways indicated in Figure~\ref{fig1}. We shall try to establish a skein relation satisfied by the invariant $\Delta_{d, S}$. Indeed, by the Alexander theorem we may assume that $L_+$ is in braided form and that $L_+ = \widehat{\beta \sigma_i}$ for some $\beta\in B_n$.  Also that $L_- = \widehat{\beta \sigma_i^{-1}}$ and that $L_0 = \widehat{\beta}$. Apply now relation Eq.~\ref{invrs} for the $g_i^{-1}$ in the expression below, noting that $\epsilon(\beta \sigma_i^{-1}) = \epsilon(\beta)-1$ and $\epsilon(\beta \sigma_i) = \epsilon(\beta)+1$:
\begin{eqnarray*}
\Delta_{d, S} (L_{-})  & = &  D^{n-1} (\sqrt{\lambda})^{\epsilon(\beta \sigma_i^{-1})} {\rm tr}_d (\flat_{d,n}(\beta) g_i^{-1}) \\
 &  = &
D^{n-1} (\sqrt{\lambda})^{\epsilon(\beta)-1}\big[
 {\rm tr}_d (\flat_{d,n}(\beta) g_i) - (u^{-1} - 1)\, {\rm tr}_d (\flat_{d,n}(\beta) e_{d,i}) \\
 &  &
 + (u^{-1} - 1)\, {\rm tr}_d (\flat_{d,n}(\beta) e_{d,i} \, g_i)\big] \\
 & = &
\frac{1}{\lambda} \Delta_{d, S} (L_{+})  - D^{n-1} (\sqrt{\lambda})^{\epsilon(\beta)-1} (u^{-1} - 1)\, {\rm tr}_d (\flat_{d,n}(\beta) e_{d,i}) \\
& &
+ D^{n-1} (\sqrt{\lambda})^{\epsilon(\beta)-1} (u^{-1} - 1)\, {\rm tr}_d (\flat_{d,n}(\beta) e_{d,i} \, g_i).
\end{eqnarray*}
The problem is that the algebra words $\flat_{d,n}(\beta) e_{d,i}$ and $\flat_{d,n}(\beta) e_{d,i} \, g_i$ do not have a natural lifting in the braid groups, even if we break the $e_{d,i}$'s according to Eq.~\ref{edi}. This was not the case in \cite{jula2}, where we were dealing with framed braids and all algebra generators had natural liftings in the framed braid groups. Yet, we have in the algebra ${\rm Y}_{d,n}$ the following `closed' relation (compare with \cite{fu}).

\begin{lem}\label{cubic}
The generators $g_i$ of the Yokonuma--Hecke algebra ${\rm Y}_{d,n}$ satisfy the cubic relations:
\begin{equation}\label{cubic1}
g_i^3 = -u g_i^2 + g_i +u
\end{equation}
Equivalently,
\begin{equation}\label{cubic2}
g_i^{-1}  = u^{-1} g_i^2 + g_i - u^{-1}
\end{equation}
\end{lem}

\begin{proof}
>From Lemma~\ref{powers} we find the relation $g_i^3 = g_i + (u-1) e_{d,i} g_i -(u-1) e_{d,i} g_i^2$. Substituting Eq.~\ref{quadr} and replacing the expression $(u-1) e_{d,i} -(u-1) e_{d,i} g_i$ by the expression  $g_i^2 - 1$ we arrive at the stated cubic relation.
\end{proof}

\smallbreak
\begin{figure}[H]
\begin{center}
 
\begin{picture}(230,80)
\put(29,80){\line(-1,-1){28}}
\put(17,63){\line(1,-1){11}}
\put(0,80){\line(1,-1){11}}

\put(28,52){\vector(-1,-1){28}}
\put(18,35){\vector(1,-1){11}}
\put(0,52){\line(1,-1){11}}

%%%%%%%%%%

\put(120,80){\vector(-1,-1){50}}
\put(98,52){\vector(1,-1){22}}
\put(70,80){\line(1,-1){22}}

%%%%%%%%%

\put(160,80){\vector(0,-1){50}}
\put(190, 80){\vector(0,-1){50}}

 %%%%%%%%%%%

\put(258,58){\line(1,1){22}}
\put(252,52){\vector(-1,-1){22}}
\put(230,80){\vector(1,-1){50}}

%%%%%%%
\put(12,10){$L_{+ +}$}
\put(90,10){$L_{+}$}
\put(170,10){$L_{0}$}
\put(250,10){$L_{-}$}

\end{picture}
\caption{$L_{++}$, $L_+$, $L_0$ and $L_-$}\label{fig1}
\end{center}
\end{figure}
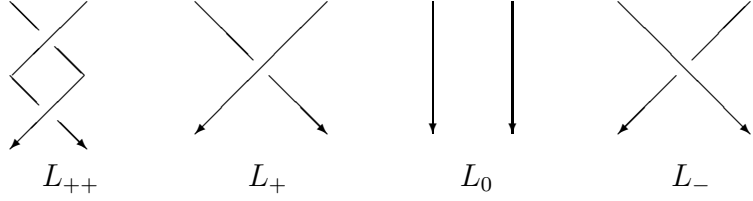

We then have the following result.

\begin{prop}
The invariant $\Delta_{d, S}$ satisfies the following cubic skein relation:
\begin{equation}\label{skein}
\sqrt{\lambda}\, \Delta_{d, S} (L_{-}) =
\frac{1}{\lambda u}\, \Delta_{d, S} (L_{++})  + \frac{1}{\sqrt{\lambda}}\, \Delta_{d, S} (L_{+})
- \frac{1}{u}\, \Delta_{d, S} (L_{0}).
\end{equation}
\end{prop}

\begin{proof}
By the same reasoning as above we may assume that $L_0 = \widehat{\beta}$ for some $\beta\in B_n$. Also that $L_+ = \widehat{\beta \sigma_i}$, $L_{++}  = \widehat{\beta \sigma_i^2}$ and $L_- = \widehat{\beta \sigma_i^{-1}}$. Apply now Eq.~\ref{cubic2} from Lemma~\ref{cubic} in the expression below, noting that $\epsilon(\beta \sigma_i^{-1}) = \epsilon(\beta)-1$, $\epsilon(\beta \sigma_i) = \epsilon(\beta)+1$ and $\epsilon(\beta \sigma_i^2) = \epsilon(\beta)+2$.
\begin{eqnarray*}
\Delta_{d, S} (L_{-}) &  = &  D^{n-1} (\sqrt{\lambda})^{\epsilon(\beta \sigma_i^{-1})} {\rm tr}_d (\flat_{d,n}(\beta) g_i^{-1}) \\
&  = &
D^{n-1} (\sqrt{\lambda})^{\epsilon(\beta)-1}
\left[ u^{-1}\, {\rm tr}_d (\flat_{d,n}(\beta) g_i^2) + {\rm tr}_d (\flat_{d,n}(\beta) g_i) -
u^{-1}\, {\rm tr}_d (\flat_{d,n}(\beta)) \right] \\
&  = &
\frac{1}{(\sqrt{\lambda})^3 u} \Delta_{d, S} (L_{++})  + \frac{1}{\lambda} \Delta_{d, S} (L_{+})
- \frac{1}{\sqrt{\lambda} u} \Delta_{d, S} (L_{0}).
\end{eqnarray*}
\end{proof}

\subsection{\it An isotopy invariant from ${\rm tr}_{\infty}$}

In this subsection we extend the values of the invariants $\Delta_{d, S}$ to the adelic context. By Eq.~\ref{adelicrepbn} the braid group $B_n$ is represented in ${\rm Y}_{\infty,n}=  \varprojlim_{d\in {\Bbb N}^{\sim}} {\rm Y}_{d,n}$ via the map
$\flat_{\infty,n} = \varprojlim_{d\in {\Bbb N}^{\sim}} \flat_{d,n}$. Further, by Theorem~\ref{ptrace}, elements in ${\rm Y}_{\infty,n}$ map, via the Markov trace ${\rm tr}_{\infty,n} =  \varprojlim_{d\in {\Bbb N}^{\sim}} {\rm tr}_{d,n}$, in the ring $\varprojlim_{d\in {\Bbb N}^{\sim}} R[X_d]$, where $R = {\Bbb C}[z]$.

 For any $d\vert d^{\prime}$, now, the connecting ring epimorphism $\xi_d^{d^{\prime}}$ (recall Eq.~\ref{xi}) yields a connecting epimorphism $\Xi_d^{d^{\prime}}$ from the ring of rational functions ${\Bbb C}(z,X_{d^{\prime}})$ to the ring of rational functions ${\Bbb C}(z,X_d)$.

\begin{lem}\label{gammalift}
The following diagram is commutative.
\begin{equation}
\begin{diagram}
\node{{\mathcal L}}
\arrow{e,t}{\Delta_{d^{\prime},S}} \arrow{s,l}{\rm Id} \node{{\Bbb C}\left(z, X_{d^{\prime}}\right)}
\arrow{s,r}{\Xi_d^{d^{\prime}}}\\
\node{{\mathcal L}} \arrow{e,b}{\Delta_{d,S}} \node{{\Bbb C}\left(z,X_{d}\right)}
\end{diagram}
\end{equation}
\end{lem}

  We shall further denote by $R_{\infty}$ the field of fractions of $\varprojlim_{d\in {\Bbb N}^{\sim}} R[{\rm X}_d]$. Taking now inverse limits in the diagram of Lemma~\ref{gammalift} we obtain the map $\Delta_{\infty,S} := \varprojlim_{d\in {\Bbb N}^{\sim}} \Delta_{d,S}$  and we have the following.

\begin{thm}\label{padicinv}
If for all $d$ the set ${\rm X}_{d}$ satisfies the $E$--condition, then the map
$$
\begin{array}{ccccl}
\Delta_{\infty,S} & : &  {\mathcal L} & \longrightarrow &R_{\infty} \\
& & \widehat{\alpha} & \mapsto & (\Delta_{d,S}(\widehat{\alpha}), \Delta_{d^{\prime},S}(\widehat{\alpha}),\ldots)
\end{array}
$$
for any $\alpha\in  \cup_nB_n$ is an isotopy invariant of oriented links in $S^3$.
Moreover:
$$
\Delta_{\infty,S}(\widehat{\alpha}) = \left(\frac{1-\lambda u}{\sqrt{\lambda}(1-u)\zeta_{d,S}}\right)^{n-1}
(\sqrt{\lambda})^{\epsilon(\alpha)} {\rm tr}_{\infty} (\flat_{\infty,n}(\alpha)) = D^{n-1} (\sqrt{\lambda})^{\epsilon(\alpha)}  {\rm tr}_{\infty} (\flat_{\infty,n}(\alpha)).
$$
\end{thm}

\begin{proof}
By  Lemma~\ref{cohlif} we have non--trivial solutions of the $E$--system in the adelic context.
Let now  $\beta, \alpha  \in  \cup_nB_n$ be  Markov equivalent braids. Then, any isotopy invariant  agrees on the closures $\widehat{\beta}$ and $\widehat{\alpha}$. So, $\Delta_{d,S} (\widehat{\beta}) = \Delta_{d,S} (\widehat{\alpha})$, $\Delta_{d^{\prime},S} (\widehat{\beta}) = \Delta_{d^{\prime},S} (\widehat{\alpha})$, etc. Hence: $\Delta_{\infty,S} (\widehat{\alpha}) = \Delta_{\infty,S} (\widehat{\beta})$. Moreover, we have:

\begin{eqnarray*}
\Delta_{\infty,S} (\widehat{\alpha})
& = &
(\Delta_{d,S}(\widehat{\alpha}), \Delta_{d^{\prime},S}(\widehat{\alpha}),\ldots) \\
& = &
 ( D^{n-1} (\sqrt{\lambda})^{\epsilon(\alpha)} {\rm tr}_d(\flat_{d,n}(\alpha)), D^{n-1} (\sqrt{\lambda})^{\epsilon(\alpha)} {\rm tr}_{d^{\prime}}(\flat_{d^{\prime},n}(\alpha)),\ldots)\\
& = &
 D^{n-1} (\sqrt{\lambda})^{\epsilon(\alpha)} ({\rm tr}_d(\flat_{d,n}(\alpha)), {\rm tr}_{d^{\prime}}(\flat_{d^{\prime},n}(\alpha)), \ldots)\\
& = &
 D^{n-1} (\sqrt{\lambda})^{\epsilon(\alpha)}  {\rm tr}_{\infty} (\flat_{\infty,n}(\alpha)).
 \end{eqnarray*}
\end{proof}

The link invariant $\Delta_{\infty,S}$ is an adelic extension of the Jones polynomial.

%%%%%%%%%%%%%%%%%%%%%%%%%%%%%%%%%%%%%%%%%%%%%%%%%%%%%%%%%%%%%%%%%%%%%%%%%%%%%%%%%%%%%%%%%%%%%%%%%%%%%


\begin{thebibliography}{50}

\bibitem{fu} L. Funar, {\em On the quotients of cubic Hecke algebras},
Commun. Math. Phys. {\bf 173} (1995), 513--558.

\bibitem{jo} V.F.R. Jones, {\em Hecke algebra representations of braid groups and link
polynomials}, Ann. Math. {\bf 126} (1987), 335--388.

\bibitem{ju} J. Juyumaya, {\em Markov trace on the Yokonuma--Hecke algebra},
J. Knot Theory and its Ramifications {\bf 13} (2004), 25--39.

\bibitem{jula} J. Juyumaya, S. Lambropoulou, {\em $p$--adic framed braids},
Topology and its Applications {\bf 154} (2007), 1804--1826.

\bibitem{jula2} J. Juyumaya, S. Lambropoulou, {\em $p$--adic framed braids II}, submitted for publication. See  ArXiv:0905.3626v1 [math.GT] 22 May 2009.

\bibitem{jula3} J. Juyumaya, S. Lambropoulou, {\em An invariant for singular knots},
J. Knot Theory and its Ramifications, {\bf 18}, No. 6 (2009) 825--840.

\bibitem{la} S. Lambropoulou, {\em Knot theory related to generalized and cyclotomic
Hecke algebras of type {\rm B}}, J. Knot Theory and its Ramifications {\bf 8},
No. 5 (1999), 621--658.

\bibitem{riza} L. Ribes and P. Zalesskii, {\em Profinite Groups}, A Ser. Mod. Sur. Math. 40,
Springer, 2000.

\bibitem{wi} Wilson, {\em Profinite Groups}, London Math. Soc. Mono., New Series 19,
Oxford Sc. Publ., 1998.

\bibitem{yo} T. Yokonuma, {\em Sur la structure des anneaux de Hecke d' un
groupe de Chevallley fini}, C.R. Acad. Sc. Paris, {\bf 264} (1967),  344--347.

%%%%%%%%%%%%%%%%%%%%%%%%%%%%%%%%%%%%%%%%%%%%%%%%%%%%%%%%%%%%%%%%%%%%%%%%%%%%%%%%%%%%%%%%%%%%%%%%%%%%%



\end{thebibliography}
\end{document}